\documentclass[11pt]{article}
\usepackage{amsmath,amsthm,amsfonts,amssymb,mathrsfs,bm,graphicx,amscd,epsfig,psfrag,verbatim,comment,hyperref}
\usepackage[usenames]{color}
\parskip 	\bigskipamount

\newtheorem{theorem}{Theorem}[section]

\newtheorem{corollary}[theorem]{Corollary}
\newtheorem{proposition}[theorem]{Proposition}

\newtheorem{definition}{Definition}

\newtheorem{notation}{Notation}

\def\Z{\mathbb{Z}}

\def\R{\mathbb{R}}
\def\C{\mathbb{C}}

\def\P{\mathbb{P}}
\def\e{\epsilon}

\def\G{\mathcal{G}}
\def\F{\mathcal{F}}

\def\L{\Lambda}
\def\g{\Gamma}
\def\gm{\gamma}
\def\t{\theta}
\def\la{\lambda}
\def\k{\mathbb{K}}
\def\D{\mathcal{D}}
\def\eps{\epsilon}

\def\o{\omega}
\def\z{\zeta}
\def\c{\mathcal{C}}
\def\j{\mathfrak{g}}
\def\tj{\Phi(\mathfrak{g})}
\def\e{\mathcal{E}}

\def\b{\mathcal{B}}
\def\el{\mathcal{L}}
\def\S{\mathcal{S}}
\def\T{\mathcal{T}}
\def\J{\mathfrak{J}}
\def\I{\mathcal{I}}
\def\k{\kappa}
\def\u{\Upsilon}
\def\rr{\mathcal{R}}

\def\nat{\mathbb{N} \cup \{0\}}
\def\uz{\underline{\zeta}}
\def\uout{\Upsilon_{\out}}
\def\Tau{\mathcal{T}}
\def\out{\mathrm{out}}
\def\iin{\mathrm{in}}
\def\els{\mathcal{L}_{\Sigma}}
\renewcommand{\l}[0]{\left }
\renewcommand{\r}[0]{\right}

\makeatletter
\renewcommand*{\@cite@ofmt}{\hbox}
\makeatother

\begin{document}

\title{\bf Continuum Percolation for Gaussian zeroes and Ginibre eigenvalues }
\author{
\begin{tabular}{c}
\href{http://math.berkeley.edu/~subhro/}{Subhroshekhar Ghosh}\\ Department of Mathematics\\ University of California, Berkeley\\ subhro@math.berkeley.edu
\end{tabular}
\and 
\begin{tabular}{c}
\href{http://math.iisc.ernet.in/~manju/}{Manjunath Krishnapur}\\ Department of Mathematics\\ Indian Institute of Science\\ manju@math.iisc.ernet.in
\end{tabular}
\and 
\begin{tabular}{c}
\href{http://research.microsoft.com/en-us/um/people/peres/}{Yuval Peres}\\ Theory Group\\ Microsoft Research, Redmond\\ peres@microsoft.com
\end{tabular}
}
\date{}
\maketitle
\thispagestyle{empty}

\begin{abstract}
We study continuum percolation on certain negatively dependent point processes on $\R^2$. Specifically, we study the Ginibre ensemble and the planar Gaussian zero process, which are the two main natural models of translation invariant point processes on the plane exhibiting local repulsion. For the Ginibre ensemble, we establish  the uniqueness of infinite cluster in the supercritical phase. For the Gaussian zero process, we establish that a non-trivial critical radius exists, and we prove the uniqueness of infinite cluster in the supercritical regime.
\end{abstract}

\newpage
\section{Introduction}
\label{intro}
Let $\Pi$ be a simple point process in Euclidean plane. We place open disks of the same radius $r$ around each point of $\Pi$, and say that two points are neighbours if the corresponding disks overlap. Two points in $\Pi$ are connected if there is a sequence of neighbouring points of $\Pi$ that include these two points. We can then study the statistical properties of the maximal connected components (referred to as ``clusters'') of the points of $\Pi$. Of particular interest are the infinite cluster(s).  This is the basic setting of the continuum percolation model, also referred to as the Boolean model. 

It is clear from an easy coupling argument that the probability that an infinite cluster exists is an increasing function of the radius of the disks. We say that there is a non-trivial critical radius if there exists an $0<r_c<\infty$ such that the probability of having an infinite cluster is zero when $0<r<r_c$ and the same probability is is strictly positive when $r_c<r<\infty$.  For $r_c<r<\infty$,  one can ask whether the infinite cluster is unique. For point processes which are ergodic under the action of translations, the event that there is an infinite cluster is translation-invariant, and therefore  its probability is either 0 or 1. Similarly, the number of infinite clusters is a translation-invariant random variable, and therefore a.s. a constant.

In this paper we focus on the two main natural examples of repelling point processes on the plane: the Ginibre ensemble, arising as weak limits of certain random matrix eigenvalues, and the Gaussian zero process arising as weak limits of zeroes of certain random polynomials. The latter process will be abbreviated as the GAF zero process.
For details on these models, see Section \ref{models}.

In \cite{BY} (see Corollary 3.7 and the discussion thereafter) it has been shown that there exists a non-zero and finite critical radius for the Ginibre ensemble. 

In this paper we prove the following theorems:

\begin{theorem}
\label{gin}
In the Boolean percolation model on the Ginibre ensemble, a.s. there is exactly one infinite cluster in the supercritical regime.
\end{theorem}

\begin{theorem}
\label{gaf}
In the Boolean percolation model on the GAF zero process, there exists a non-zero and finite critical radius. Moreover,  in the supercritical regime,  a.s.  there is exactly one infinite cluster.
\end{theorem}

Continuum percolation is  well-studied in theoretical and applied probability, as a model of communication networks, disease-spreading through a forest, and many other phenomena. This model, also referred to as the Gilbert disk model or the Boolean model,  is almost as old as the more popular discrete bond percolation theory. It was introduced by Gilbert in 1961 \cite{Gil}. In the subsequent years, it has been studied extensively by different authors, such as \cite{Ha}, \cite{Mo},\cite{MR} and \cite{Pe}, among others. Closely related models such as random geometric graphs, random connection models, face percolation in random Voronoi tessellations have also been studied. For a detailed discussion of continuum percolation and related models, we refer the reader to \cite{MR} and \cite{BoRi}. For further details on point processes, we refer to \cite{DV}.

Much of the literature so far has focused on studying continuum percolation where the underlying point process $\Pi$ is either a Poisson process or a variant thereof. Most of these models exhibit some kind of spatial independence. This property is extremely useful in the study of continuum percolation on these models. E.g., the spatial independence enables us to carry over Peierls type argument from discrete percolation theory for establishing phase transitions in the existence of infinite clusters, or Burton and Keane type arguments in order to prove uniqueness of infinite clusters. 

While the Poisson process is the most extensively studied point process, the spatial independence built into it makes it less effective as a model for many natural phenomena. This makes it of interest to study point processes with non-trivial spatial correlation,  particularly those where the points exhibit repulsive behaviour. On the complex plane, the main natural examples of translation-invariant point processes exhibiting repulsion are the Ginibre ensemble and the Gaussian zero process. The latter process is also known as the Gaussian analytic function (GAF) zero process. The former arises as weak limits of eigenvalues of (non-Hermitian) random matrices, while the latter arises as weak limits of zeros of Gaussian polynomials. For precise definitions of these processes we refer the reader to Section \ref{models}.

The Ginibre ensemble was introduced by the physicist Ginibre \cite{Gin} as a physical model based on non-Hermitian random matrices. In the mathematics literature it has been studied by \cite{RV} and \cite{Kr} among others. The Gaussian zero process also has been studied in either field, see,  e.g. \cite{BoBL}, \cite{STs1},\cite{STs2},\cite{STs3},\cite{NSV}, \cite{FH}. We refer the reader to \cite{NS} for a survey. These models are distinguished elements in broader classes of repulsive point processes. For example, the Ginibre ensemble is essentially the unique determinantal process on the plane whose kernel $K(z,w)$ is holomorphic in the first variable, and conjugate holomorphic in the second (\cite{Kr}). The Gaussian zero process is essentially unique (up to scaling) among the zero sets of  Gaussian power series in that its distribution is invariant under  translations (\cite{STs1}). For an exposition on both the processes, we refer the reader to \cite{HKPV}.

The strong spatial correlation present in the above models severely limits the effectiveness of standard independence-based arguments from the Poisson setting while studying continuum percolation. Our aim in this paper is to study continuum percolation on the two natural models of repulsive point processes mentioned above, and establish the basic results. Namely, we prove that there is indeed a non-trivial critical radius for these processes, and the infinite cluster is unique when we are in the supercritical regime. 

While the spatial independence of the Poisson process is not available in these models, we observe that this obstacle can be largely overcome if we can obtain detailed understanding of spatial conditioning in these point processes. Recently, such understanding  has been obtained in \cite{GNPS}, where it has been shown that for a given domain $\D$, the point configuration outside $\D$ determines a.s. the number of points in $\D$ (in the Ginibre case) and their number and the centre of mass (in the Gaussian zeroes case), and ``nothing further''. For a precise statement of the results, we refer the reader to the Theorems \ref{gin-1}, \ref{gin-2}, \ref{gaf-1} and \ref{gaf-2} quoted in this paper. In the present work, we demonstrate that along with certain estimates on the strength of spatial dependence, this understanding is sufficient to overcome the problem of lack of independence, and answer the basic questions in continuum percolation on these two processes. 

For determinantal point processes in Euclidean space, it is known that a non-trivial critical radius exists, see, e.g. \cite{BY}. This covers the  Ginibre ensemble.  The uniqueness of the infinite cluster (in the supercritical regime), however, was not known, and this is proved in Section \ref{uqgin}. For the Gaussian zero process, both the  existence of a non-trivial critical radius and the uniqueness of the infinite cluster (when one exists) are new results, and are established in Sections \ref{gafzeros} and \ref{uqgaf} respectively.

In the case of the GAF zero process, while proving our main results we derive  new estimates for hole probabilities. Let $B(0;R)$ be the disk with centre at the origin and radius $R$. The hole probability for $B(0;R)$ is the probability $p(R)=\P \l(B(0;R) \text{ has no GAF zeroes } \r)$. It has been studied in detail in \cite{STs3}, and culminated in the work of Nishry \cite{Nis} where he obtained the precise asymptotics as $R \to \infty$.  It turns out that as $R \to \infty$ we have $ -\log p(R)/R^4  \to c$ where $c>0$ is a constant. In this paper, however, we need to understand  hole probabilities for much more general sets than disks. 

Let us divide the plane into $\t \times \t$ squares given by the grid $\t \Z^2$. Let $\g(L)$ be a connected set comprising of $L$ such squares. We prove that for $\t \ge \t_0$ (an absolute constant), there is a quantity $c(\t) > 0$ such that $\P\l(\g(L) \text{ has no GAF zeroes }\r) \le \exp \l( -c(\t)L \r)$. The techniques generally used in the literature to study hole probabilities, e.g. Offord-type estimates, do not readily apply to this situation. Instead, we exploit the almost independence property of GAF, and combine it with a Cantor set type construction to obtain the desired result. 

\subsection{The Boolean model}
\label{bm}
Let $\Pi$ be a point process in $\R^2$ whose one-point and two-point intensity measures are absolutely continuous with respect to the Lebesgue measure on $\R^2$ and $\R^2 \times \R^2$ respectively. We say two points $x,y$ of $\Pi$ are \textsl{neighbours} of each other if $\| x-y \|_2 < 2r$. Equivalently, we can place open disks of radius $r$ around each point; then two points are connected if and only the corresponding disks intersect.  Two points $x,y$ of $\Pi$ are \textsl{neighbours} if $\exists$ a finite sequence of points $x_0,x_1,\cdots,x_n \in \Pi$ such that $x_0=x,x_n=y$ and $x_{j+1}$ is the neighbour of $x_j$ for $0 \le j \le n-1$.  

This is the Boolean percolation model on the point process $\Pi$ with radius $r$, denoted by $X(\Pi,r)$.

Connectivity as defined above is an equivalence relation, and the maximal connected components are called clusters. The size of a cluster is the number of points of $\Pi$ in that cluster. We say that the model percolates if there is at least one infinite cluster. We say that $x_0 \in \R^2$ is connected to the infinity if there is a point $x\in \Pi$ such that $\|x-x_0\|_2 < r$ and $x$ belongs to an infinite cluster. The probability of having an infinite cluster and that of origin being connected to infinity both depend on the parameter $r$. The trivial coupling obtained from the inclusion of a disk of radius $r'$ inside a disk of radius $r$ with the same centre (for $r'<r$) shows that both of these probabilities are non-decreasing in $r$. 

\begin{notation}
\label{lar}
Let $\L(r)$ denote the number of infinite clusters when the disks are of radius $r$.
\end{notation}

\begin{definition}
The point process $\Pi$ is said to have a critical radius $0<r_c<\infty$ if  $\L(r)=0$ a.s. when $0<r<r_c$ and $\P \l( \L(r)>0 \r) >0$ when $r_c<r<\infty$.
\end{definition}

For any point process $\Pi$ in $\R^2$, the group of translations of $\R^2$ acts in a natural way on $\Pi$: a translation $T$ takes the point $x \in \Pi$ to $T(x)$, the resulting point process being denoted $T_*{\Pi}$. The process $\Pi$ is said to be translation invariant if $T_{*}\Pi$ has the same distribution as $\Pi$ for all translations $T$. The process is said to be ergodic under translations if this action is ergodic. 

For any translation invariant point process, the probability of the origin being connected to infinity is the same as that for any $x \in \R^2$, so by a simple union bound over $x \in \R^2$ with rational co-ordinates, the probability of having an infinite cluster is positive if and only if the probability of the origin being connected to infinity is positive. 

Clearly, $\L(r)$ is a translation-invariant random variable.   If $\Pi$ is ergodic, $\L(r)$  is a.s. a constant $\in \Z_{\ge 0}$.  In particular, the probability of having an infinite cluster is either 0 or 1.

\subsection{The underlying graph}
\label{ug}
Consider the Boolean model with radius $r$ on a point process $\Pi$ in $\R^2$. By the \textsl{underlying graph} $\j$ of this model we mean the graph whose vertices are the points of $\Pi$ and two vertices $x,y$ are neighbours iff $\| x-y \|_2 < 2r $. By $\tj$ we denote the subset of $\R^2$ formed by the union of  the points of $\Pi$ and straight line segments drawn between two such points whenever their mutual distance is less than $2r$. Since the two point intensity measure of $\Pi$ is absolutely continuous with respect to the Lebesgue measure on $\R^2 \times \R^2$, therefore the probability that there are two points of $\Pi$ at a mutual distance  $2r$ is 0. Hence, if a take a large open disk $D$, then there exists an $\eps > 0$ such that for each point $x$ of $\Pi$ in $D$ we have $\b(x;\eps) \subset D$ and $x$ can be moved to any new position in the open disk $\b(x;\eps)$ without changing the connectivity properties of $\j$.   

\section{The Models}
\label{models}
\subsection{The Ginibre Ensemble}
Let us consider an $n \times n$ matrix $X_n, n\ge 1$ whose entries are i.i.d. standard complex Gaussians. 
 The vector of its eigenvalues, in uniform random order, has the joint density (with respect to the Lebesgue measure on $\C^n$) given by
\[
 \label{gindef}
p(z_1,\cdots,z_n)= \frac{1}{\pi^n\prod_{k=1}^n k!}  e^{-\sum_{k=0}^n |z_k|^2} \prod_{i<j} |z_i-z_j|^2
\]
Recall that a determinantal point process on the Euclidean space $\R^d$ with kernel $K$ and background measure $\mu$ is a point process on $\R^d$ whose $k$-point intensity functions with respect to the measure $\mu^{\otimes k}$ are given by \[ \rho_k(x_1,\cdots,x_k)= \text{det} \bigg[  \l( K(x_i,x_j) \r)_{i,j=1}^k  \bigg] \]
Typically, $K$ has to be such that the integral operator defined by $K$ is a non-negative trace class contraction mapping $L^2(\mu) \to L^2(\mu)$. For a detailed study of determinantal point processes, we refer the reader to \cite{HKPV} or \cite{Sos}. A simple calculation involving Vandermonde determinants shows that the eigenvalues of $X_n$
(considered as a random point configuration) form a determinantal point process in $\R^2$. Its kernel is given by $ K_n(z,w)=\sum_{k=0}^{n-1} \frac{(z \bar{w})^k}{k!}$ with respect to the background measure $ \, d\gamma (z) = \frac{1}{\pi}e^{- |z|^2} d\mathcal{L}(z)$ where $\mathcal{L}$ denotes the Lebesgue measure on $\C$. 
This point process is the Ginibre ensemble (of dimension $n$), which we will denote by $\G_n$.
As $n\rightarrow \infty$, these point processes converge, in distribution, to a determinantal point process given by the kernel $ K(z,w)=e^{z\bar{w}}=\sum_{k=0}^{\infty} \frac{(z \bar{w})^k}{k!}$ with respect to the same background measure $\gamma$. This limiting point process is the infinite Ginibre ensemble $\G$. It is known that $\G$ is ergodic under the natural action of the translations of $\R^2$.

\subsection{The GAF zero process}
Let $\{\xi_k \}_{k=1}^{\infty}$ be a sequence of i.i.d. standard complex Gaussians. Define, for $n \ge 0$, 
\[
f_n(z) =\sum_{k=0}^{n}\xi_k \frac{z^k}{\sqrt{k!}} \, \, , \qquad \qquad f(z) =\sum_{k=0}^{\infty}\xi_k \frac{z^k}{\sqrt{k!}} \, \, .
\]
These are complex Gaussian processes on $\C$ with covariance kernels given by 
\[ K_n(z,w)=\sum_{k=0}^{n} \frac{(z \bar{w})^k}{k!} \qquad \text{and} \qquad K(z,w)=\sum_{k=0}^{\infty} \frac{(z \bar{w})^k}{k!}\]  respectively. A.s. $f$ is an  entire function and the $f_n$-s converge  to $f$ (in the sense of the uniform convergence of functions on compact sets). It is not hard to see (e.g., via Rouche's theorem) that this implies that the corresponding point processes of zeroes, denoted by $\F_n$, converge a.s. to the zero process $\F$ of the GAF (in the sense of locally finite point configurations converging on compact sets). It is known that $\F$ is ergodic under the natural action of the translations of the plane.

\section{Discrete approximation and the critical radius}
The first step in our study of continuum percolation will be to relate our events of interest to  events defined with respect to a grid, so that the problem becomes amenable to techniques similar to the ones that are effective in studying percolation in discrete settings.  

\begin{definition}
\label{bl}
Let $\t>0$ be a parameter, to be called  base length, and consider the grid $\t\Z^2$. Each $\t \times \t$ closed square obtained from the edges of this grid will be referred to as a \textsl{standard square}.  Two (distinct) standard squares are said to be neighbours if their boundaries intersect. So, each standard square has 8 neighbours. 
\end{definition}

\begin{notation}
\label{misc}
For $x \in \R^2$ and $R>0$, we will denote by $\b(x;R)$ the open disk with centre $x$ and radius $R$.

We define $W_R$, the box of size $R$, to be the set $W_{R}:=\{x\in \R^2 : \| x \|_{\infty} = R \}$.

For a subset $K \subset \R^2$, we will denote by $\overline{K}$ the topological closure of $K$. 
\end{notation}

\begin{definition}
\label{defpath}
Fix a radius  $r>0$ and a base length $\t >0$. 

A continuum path  $\gm$ of length $n$ is defined to be a piecewise linear curve whose vertices are given by the  sequence of points $x_j \in \Pi, 1 \le j \le n$ such that $x_{i+1}$ is a neighbour of $x_i$ for $1\le i \le n-1$.  

For a continuum path $\gm$ with vertices $\{x_1,\cdots,x_n\}$, we denote by $S(\gm)$ the set $\bigcup_{i=1}^{n} \b(x_i;r)$.

A lattice path $\g$  of length $n$ is defined to be a sequence of standard squares $\{X_j\}_{j=1}^n$ such that $X_{i+1}$ is a neighbour of $X_i$ for $1\le i \le n-1$. A lattice path $\{X_i\}_{i=1}^n$ is said to be non-repeating if $X_i \ne X_j$ for $i \ne j$. 

For a lattice path $\g=\{X_1,\cdots,X_n\}$, we denote by $V(\g)$ the set $\bigcup_{i=1}^{n} X_i$.

We say that a continuum path $\gm$ \textsl{connects} the origin to $W_R$ if $0 \in S(\gm)$ and $S(\gm) \cap W_R \ne \varphi$. For $R \in \Z_{+}$, we say that a lattice path $\g$ connects the origin to $W_{R\t}$ if $0 \in V(\g)$ and $V(\g) \cap W_{R\t} \ne \varphi$.
\end{definition}

With these notions in hand, we are ready to state:

\begin{proposition}
 \label{discr1}
Consider the Boolean percolation model $X(\Pi,r)$.  Let the base length $\t=r/\sqrt{5}$. Suppose, for some $L \in \Z_{+}$, there exists a non-repeating lattice path $\g=\{X_1,\cdots,X_n\}$ that connects 0 to $W_{L\t}$ with each $X_i$ containing at least one point in $\Pi$. Then there exists a continuum path $\gm$ that connects 0 to $W_{L\t}$. 
\end{proposition}

\begin{proof}
The result follows from the fact that with $\t= r/\sqrt{5}$, disks of radius $r$ centred at points in adjacent standard squares intersect with each other. 
\end{proof}

\begin{proposition}
 \label{discr2}
Fix a base length $\t$ and an integer $k \ge 0$. For any $0<r<\t/18 k$ the following happens:  
Suppose there exists a continuum path $\gm$ connecting 0 to $B_{L\t}$ (where $L \in \Z_{+}$). Then there exists a non-repeating lattice path $\g$ connecting 0 to $B_{L\t}$ such that each standard square in $\g$ contains $\ge k$ points $\in \Pi$.
\end{proposition}
\begin{proof}
Fix a radius $r$ such that $k < \t/{18r}$, and  let $\gm$ be a continuum path with vertices $\{x_i\}_{i=1}^n$ connecting 0 to $B_{L\t}$. A finite lattice path $\g_1$ is said to be contained in another finite lattice path $\g_2$, denoted by $\g_1 \subset \g_2$, if $V(\g_1)\subset V(\g_2)$. 

Now, consider the set $\Xi$ of all finite lattice paths  $\g$ (non-repeating or otherwise), $0 \in {V(\g)}$, such that each standard square in $\g$ contains $\ge k$ points. Clearly, $\subset$ is a partial order on $\Xi$. Moreover, $\Xi$ is non-empty, because $\gm$ must reach $L^{\infty}$ distance $\t$ from the origin, and in doing so must have at least $\t/2r$ points $\in \Pi$. The $4$ standard squares whose closures contain the origin contain these $\t/2r$ points, so at least one of them must have at least $\t/8r \ge k$ points in $\Pi$. 

Let $\g$ be a maximal element in $\Xi$ under $\subset$.  If $\g$ connects 0 to $B_{L\t}$ then we are done. Otherwise, we define the \textsl{surround} $\Sigma(\g)$ of $\g$ as the  union  of all standard squares which are neighbours of the standard squares in $\g$ and are contained in the unbounded component of the complement of $\g$. Since $\gm$  connects 0 to $B_{L\t}$, therefore $\gm$ intersects $\partial \Sigma(\g) \setminus V(\g)$. Let $j$ be the least index $\in [n]$ such that the line segment $(x_{j-1},x_j]$ intersects $\partial \Sigma(\g) \setminus V(\g)$. Since $r < \t$, we must have $x_{j-1} \in \text{ Int } \l(\Sigma(\g)\r)$, where $\text{ Int }(H)$ denotes the interior of a set $H$. Let $\sigma$ be a standard square in $\Sigma(\g)$ such that ${\sigma}$ contains $x_{j-1}$. Consider the continuum path $\gm'$ with vertices $\{x_{j},x_{j-1},\cdots,x_{i}\}$ where $i$ is the largest index $\le j-1 $ such that $x_i \in \g$. Now the part of  $\gm'$ contained in $\sigma$ and its neighbouring standard squares (which are also in $\Sigma(\g)$) is of length at least $\t$, therefore it has at least $\t/2r$ points $\in \Pi$ contained in these squares. Since the total number of such squares (including $\sigma$) $\le 9$, we have $\t/2r$ points  $\in \Pi$ contained in $\le 9$ squares in $\Sigma(\g)$. Therefore at least one square $\sigma'$ in $\Sigma(\g)$ has at least $\t/18r \ge k$ points of $\gm$. Let $\g=\{X_i\}_{i=1}^N$ and let $\sigma'$ be a neighbour of $X_j \in \g$. We define a new lattice path $\g' \in \Xi$  by $\g'=\{ X_1,X_2\cdots,X_N,X_{N-1},X_{N-2}\cdots,X_j,\sigma \}$, that is, by backtracking along $\g$ until we reach $X_j$ and then appending $\sigma$ at the end. Clearly, $S(\g') \supset S(\g)$ as a proper subset, contradicting the maximality of $\g$. 

Since the procedure described above must terminate after finitely many steps because $B_{L\t}$ is a compact set, a maximal element $\g$ of $\Xi$ must connect 0 to $B_{L\t}$. Such a lattice path may  not be non-repeating. However, we can erase the loops in $\g$ in the chronological order to obtain a non-repeating lattice path of the desired kind that connects 0 to $B_{L\t}$.     
\end{proof}

In the next theorem, we provide some general conditions under which there exists a non-trivial critical radius for the Boolean percolation model

\begin{theorem}
\label{tech}
 Let $\Pi$ be a translation invariant and ergodic point process with the property that for any connected set $\g$ of $L$ standard squares (with base length $\theta$)  the following are true:
\begin{itemize}
\item (i)For large enough $\t$, we have \[\P \l[ \g \text{ contains no points } \in \Pi \r] \le \exp\l( -c_1(\theta)L \r)\] with $c_1(\theta) \to \infty$ as $\theta \to \infty$ \newline
\item (ii)For large enough $\t$, we have \[\P \l[ \text{ Each standard square in } \g \text{ has at least k points } \in \Pi  \r] \le  \exp \l( -c_2(\theta,k)L  \r) \] with $\lim_{\t \to \infty} \lim_{k \to \infty} c_2(\theta,k) = \infty $. 
\end{itemize}
In the Boolean percolation model $X(\Pi,r)$ on $\Pi$, let $r$ denote the radius of each disk and let $\L(r)$ denote the number of infinite clusters. Then there exists $0<r_c <\infty$ such that for $0 < r < r_c$, we have $\L(r)=0$ a.s. and for $r_c<r<\infty$ we have $\L(r)>0$ a.s. 
\end{theorem}

\begin{proof}
The proof follows a Peierl's type argument from the classical bond percolation theory, after appropriate discretization using Propositions \ref{discr1} and \ref{discr2}.
We first note that  by translation invariance, it suffices to show that $\P \l[ 0 \text{ is connected to } \infty \text{ with radius } r \r]>0 \text{ or } =0$ respectively in order to show that $\L(r)>0 \text{ or }=0$  a.s.

We want to show that for small enough $r$, there is no continuum path connecting 0 to $\infty$.  Consider possible base lengths $\theta$ so large that our hypothesis (ii) is valid. Fix base length $\t$ and $k$  a positive integer large enough such that $2\log 3 - c_2(\t,k) <0$ where $c_2$ is as in (ii). We call a non-repeating lattice path $\g$ to be $k-$full if each standard square in $\g$ contains $\ge k$ points $\in \Pi$.    By condition (ii), if there are $L$ distinct standard squares in $\g$, then the probability of $\g$ being k-full $\le \exp \l( -c_2(\t,k)L  \r) $. Since each  standard square has $\le 9$ neighbours, therefore the number of non-repeating lattice paths $\g$ containing 0 and  having $L$ standard squares  $\le 9^L$. So, \[\P \bigg[ \text{ There is a k-full lattice path of length } L \text{ containing the origin }\bigg] \le \exp \bigg((2\log 3 -c_2(\t,k))L  \bigg)\] 
The right hand side is summable in positive integers $L$, hence by Borell-Cantelli lemma, \[ \P \bigg[ \text{ There exists a k-full lattice path connecting the origin to } B_{L\t} \text{ for all } L \in \Z_{+} \bigg] =0\]
If there was a continuum path $\gm$ connecting 0 to $\infty$, then for any integer $t>0$ there will be a continuum path connecting 0 to $B_{t\t}$.
We now appeal to Proposition  \ref{discr2} for this $k$ and find an $r$ small enough such that  for any continuum path $\gm$ connecting 0 to the box $B_{t\t}$ we can find a $k$-full lattice path $\g$ connecting 0 to $B_{t\t}$. But we have already seen that a.s. there are only finitely many k-full lattice paths, which gives us a contradiction, and proves that there is no continuum path connecting 0 to $\infty$.

By translation invariance, this proves that for small enough $r$, we have $\L(r)=0$ a.s.

Next, we want to show that for large enough $r$,  with positive probability there exists a continuum path connecting 0 to $\infty$. Fix a radius $r$ in the Boolean model. The event that there exists no continuum path from 0 to $\infty$, implies by Proposition \ref{discr1} that (choosing the base length to be $\t$ as in Proposition \ref{discr1} with $\t > 2r$) there exists $L\in \Z_{+}$ such that there is no lattice path connecting the origin to $B_{L\t}$. The last statement implies that there exists a circuit of standard squares surrounding the origin such that the interiors of the standard squares in this circuit do not contain any point from $\Pi$. Therefore, it suffices to prove that the probability of this event can be made  $<1$ by choosing $r$ sufficiently large.

To this end, we recall that the number of  circuits of standard squares containing the origin and consisting of $L$ distinct standard squares is $ \exp \l( cL \r)$ for some constant $c>0$. For details on this, we refer the reader to \cite{BoRi}, Chapter 1, proof of Lemma 2. 

The probability that a specific circuit of standard squares surrounding the origin and containing $L$ standard squares is empty $\le \exp \l( -c_1(\t)L  \r)$ when base length is $\t$, which follows from condition (i) in the present theorem. Therefore, \begin{equation} \label{circuit} \P \bigg[ \text{ There exists an empty circuit surrounding the origin }  \bigg] \le \sum_{L=1}^{\infty}e^{c(d)L} e^{-c_1(\t)L} \end{equation} Now, by choosing $r$ large enough, we can make $\t$ large enough (by Proposition \ref{discr1}), so that condition (i) would imply that the right hand side of (\ref{circuit}) is less than 1. This completes the proof that when $r$ is large enough, 0 is connected to $\infty$ with positive probability. 
\end{proof}

\section{Critical radius for Gaussian zeros}
\label{gafzeros}
In this section we aim to study the Boolean model on the planar GAF zero process. First of all, we will prove an estimate on hole probabilities and overcrowding probabilities in $\F$, which is taken up in Section \ref{hprob}. It will subsequently be used to prove the existence of critical radius for the Boolean percolation model on Gaussian zeroes in Section \ref{gafperc}. 

\subsection{Exponential decay of hole and overcrowding probabilities}
\label{hprob}

The main goal of this section is to prove the following estimate on the hole and overcrowding probabilities of connected sets composed of standard squares:  

\begin{theorem}
\label{gafest}
 Let $\g$ be a  connected set composed of $L$ standard squares of side length $\t$. Let $E_k$ and $F_k$ denote the events that  that each standard square in $\g$ has no zeroes and has $ \ge k$ zeroes respectively. Then, for $\t$ bigger than a constant, we have  :\\
\[(i) \P \l[ E_k \r] \le \exp \l(-c_1(\t)L \r) \hspace{1 cm}  (ii) \P \l[ F_k \r] \le \exp\l(-c_2(\t,k)L\r)\]
where  $c_1(\t) \to \infty$ as $\t \to \infty$ and $\lim_{\t \to \infty} \lim_{k \to \infty}c_2(\t,k) = \infty$
\end{theorem}
We will  perform a certain Cantor type construction which will be used in proving Theorem \ref{gafest}. For the rest of this section, the symbol ``$\log$'' denotes logarithm to the base 2. We will first perform the construction for a straight line, and then take a product of the construction along the two axes which will result in a similar construction for a square.

We consider the normalized gaf $f^*(z)=e^{-\frac{1}{2}|z|^2}f(z)$. We will make use of the following almost independence theorem from [NS]:
\begin{theorem}
 \label{AI}
Let $F$ be a  GAF. There exists numerical constant $A>1$ with the following property. Given a family of compact sets $K_j$ in $\C$ with diameters $d(K_j)$, let $\rho_j \ge \sqrt{\log(3 + d(K_j))}$. Suppose that $A\rho_j$-neighbourhoods of the sets $K_j$ are pairwise disjoint. Then \[F^*=F_j^*+G_j^* \text{ on } K_j \] where $F_j$ are independent GAFs and for a positive numerical constant $C$ we have
\[\P \l\{ \max_{K_j} |G_j^*| \ge e^{-\rho_j^2} \r\} \le C\exp[-e^{\rho_j^2}] \]
\end{theorem}
Define the function $h(x) = 2A  \log(x/2) $, where $A$ is as in Theorem \ref{AI}.

Our construction will be parametrized by two parameters: $\t >0$ and $0<\la<1$. We take  $\t$ to be large enough and $\la$ to be small enough; the exact conditions demanded of $\t$ and $\la$ will be described as we proceed along the construction. It turns out that the resulting choice of $\t$ and $\la$ can be made to be uniform in all the other variables in the construction (like the length $L$), and it suffices to take $\la$ smaller than a constant and $\t$ large enough, depending on $\la$. To begin with, we demand that $\t$ be so large that $C\exp[-e^{(\log \t)^2}]<1$, for $C$ as in Theorem \ref{AI}, and
\begin{equation}
\label{choice}
\sqrt{\log (3 +  x\t \sqrt{2})} < \t \log x  \text{ for all } x \ge 1.
\end{equation}

\subsubsection{A Cantor type construction: straight line}
\label{Cantor-st}

Let $\g_0$ be the straight line segment $[0,L\t]$ (or a horizontal translate thereof). An obvious analogue for this construction can be carried out for a similar line segment aligned along the vertical axis.

We start with $\g_0$. Let $\Delta_0$ be the line segment $\subset \g_0$ of  length $\t h(L)$ such that $\l(\Delta_0\r)^c \cap \g_0$ consists of two line segments $\el$ and $\rr$ of equal length $\l(L\t-\t h(L)\r)/2$, situated respectively to the left and right of $\Delta_0$. Set $\g_1^1=\el$, $\g_1^2=\rr$ and $\g_1=\g_1^1 \cup \g_1^2$ . We proceed inductively as follows. For $N=\lceil \log  \la L  \rceil > j \ge 1$ where $0<\la<1$ is to be fixed later, consider $\g_j^i, 1\le i \le 2^j$.  Let $\Delta^i_j$ be the segment $\subset \g_j^i$ of length $\t h(L/2^j)$ such that $\el_i,\rr_i \subset \g_j^i$ are  segments of equal length, located  to the left and right respectively of $\Delta^i_j$. Each of $\el_i,\rr_i$ has length $\l( \text{ Length }(\g_j^i) -\t h(L/2^j)\r)/2$. Doing this for each $1\le i \le 2^j$ we have a collection of $2^j$  pairs of line segments  $\el_i$ and $\rr_i$, with a natural ordering among them along the horizontal axis from the left to the right. Following this order, we denote these segments $\g_{j+1}^i, 1\le i \le 2^{j+1}$. Finally, we define $\g_{j+1}=\cup_{i=1}^{2^{j+1}} \g_{j+1}^i $. We call $\Delta_{j}=\cup_i \Delta^i_j$ as the ``removed'' portion in the $j$-th round and $\g_{j+1}$ to be the ``surviving'' portion after the $j$-th round. 

This completes the construction for  a straight line segment $\g_0$ of length $L\t$.
Before moving on, let us make some observations about the above construction. Recall that  $N=\lceil \log  \la L  \rceil$, so in the end there are $2^{ N}$ disjoint segments, and $\la L  \le 2^N \le 2 \la L $. The length of each $\g_j^i$ is clearly bounded above by $L\t/2^j$, and the length of  $\Delta_j$ is bounded above by $2^j\t h(L/2^j)$ We upper bound the total length of the removed portion $\cup_{j=1}^{N-1} \Delta_j$  by 
\[ \sum_{i=1}^N 2^{j-1} \cdot 2A \t \log(L/2^j) = 2A(2^{N}-1)\t \log L - 2A\t ( N-1)2^{N} -2A\t \] \[= 2^{N+1}A \t \l( \log L - N \r) -2A\t \log L + 2^{N+1}A\t -2A\t \le 4A\l(\la \log \frac{1}{\la}\r)L\t + 4A\la L \t \]
where in the last step we have used  
\begin{equation}
\label{choice1}
\la L \le 2^N \le 2 \la L.
\end{equation}
By choosing $\la $ small enough, we can ensure that the total length of the removed portion is  $\le \frac{1}{2} L \t$. 

In $\g_N$, each final surviving segment $\g_N^i$ is of equal length, and since the length of $\g_N \ge \frac{1}{2} L\t$, therefore each $\g_N^i$ is of horizontal length $ \ge \frac{1}{2^{N+1}} L\t \ge \frac{\t}{4 \la}$.  

Notice that $A\rho_j$-neighbourhoods of $\g_j^i$ are disjoint where $\rho_j=\t\log(L /2^j)$. 

\subsubsection{A Cantor type construction: square}
\label{Cantor-box}

We now describe a variant of the Cantor type  construction in Section \ref{Cantor-st} for  a square $B_0$  dimension $L\t \times L\t$.  For each round, we will describe the connected components surviving at the end of that round.

We begin by noting that $B_0=\g_{0,1} \times \g_{0,2}$ where $\g_{0,i}$ are straight lines of length $L\t$ along horizontal and vertical directions respectively.  We perform the construction for a straight line segment on each of $\g_{0,1},\g_{0,2}$ and let the surviving set at the end of round $j$ be denoted by $\g_{j,1}$ and $\g_{j,2}$ respectively. Then the surviving set at the end of round $j$ for $B_0$ is given by $B_j:= \g_{j,1} \times \g_{j,2}$. Now $B_j$ clearly contains $2^{2j}$ connected components (which are in fact squares of side length $\le \t L/2^j$); call these $B_j^{i}, 1\le i \le 2^{2j}$ numbering them in any order.  By the arguments in Section \ref{Cantor-st}, it is clear that the final set $B_N$ has area at least $\frac{1}{4} \t^2L^2$, has $4^N$ connected components which are squares of side $\ge \t/4\la$. For $\la$ smaller than some absolute constant, each of these connected components contains at least one standard square of side $\t$. Moreover \[L^2\text{ Dist}(B_j^i,B_j^{i'}) \ge L^{\infty}\text{ Dist}(B_j^i,B_j^{i'})\ge 2A\t\log(L/2^j) \text{ for all } i \neq i', \text{ for each }j. \]
Then with $\rho_j=\t \log(L/2^j)$, we have that the $A\rho_j$-neighbourhoods of the $B_j^i$-s are disjoint, and by choice of $\t$ in equation (\ref{choice}) we have $\rho_j \ge \sqrt{\log (3 + \text{ Diam } (B_j^i) )}$. In obtaining the last assertion, we use the fact that $\text{ Diam } (B_j^i)  \le \t L \sqrt{2}/2^j$, recall (\ref{choice}) and (\ref{choice1}) and choose $\la$ such that $\frac{1}{2 \la} \ge 1$.

\subsubsection{Functional decomposition in the Cantor construction}
\label{func}
We can consider the sets $B_j^i$ to be arranged in the form of a tree $\T$ of depth $N$ where each vertex has 4 children (except at depth $N$). The children of the vertex $B_j^i$ are the vertices $B_{j+1}^{i'}$  where $B_{j+1}^{i'}$ are obtained by applying the $j+1$-th level of the construction on $B_j^i$.

Corresponding to the tree $\T$, we can perform a decomposition of the normalized GAF  $f^*$ using Theorem \ref{AI} above. We start with $f^*$, which we also call $f_0^*$. We apply Theorem \ref{AI} to the compact sets $B_1^i,1\le i \le 4$ to obtain i.i.d. normalized GAF s $f_{1,i}^*$ and corresponding errors $g_{1,i}^*$. These are the functions corresponding to the first level of the tree. At the next level, we perform a similar decomposition on each $f_{1,i}^*$ to obtain $f_{2,j}^*$ and $g_{2,j}^*, 1\le j \le 4^2$. So, on $B_{2}^i$ we have $f^*=f_{2,i}^*+g_{2,i}^*+g_{1,i'}^*$ where $B_{2}^i \subset B_1^{i'}$  We continue this decomposition recursively until we reach level $N$ in $\T$. At level $N$ we have $f^*=f_{N,i}^*+G_{i}^* \text{ on }B_N^i, 1\le i \le 2^N$ where the $f_{N,i}^*$  i.i.d. normalized GAFs.  The $G_i^*$ are the cumulative errors given by $G_i^*=\sum_{k=1}^N g_{k,n(k,i)}^*$  where $n(k,i)$ are such that $B_N^i \subset B_{k}^{n(k,i)}$. The $g_{j,i}^*$-s are not independent. However, for any two distinct vertices $B_j^i,B_j^{i'}$ at the same level $j$ in $\T$, the errors corresponding to the descendants of $B_j^i$ and those of $B_j^{i'}$ are independent. Thus at level $j$, the $4^j$ functions $g_{j,i}$ can be grouped into $4^{j-1}$ groups $\J_{j,i''},1\le i'' \le 4^{j-1}$ (each group consisting of $4$ functions whose vertices have the same parent at level $j-1$ in $\T$). Thus, the index $i''$ in $\J_{j,i''}$ can be thought to be varying over the vertices in $\T$ at level $j-1$. Clearly, $\J_{j,i}$ and $\J_{j,i'}$ are independent sets of functions for $i \neq i'$. We call $\J_{j,i}$ to be ``good'' if each $g_{j,k} \in \J_{j,i}$ satisfies $\{ \max_{B_j^k} |g_{j,k}^*| \le e^{-\rho_j^2} \}$, otherwise we call it ``bad''. Recall from Theorem \ref{AI} (and a simple union bound) that $\P \{ \J_{j,i} \text{ is bad} \} \le C\exp[-e^{\rho_j^2} ]$.

Set $p_j=C\exp[-e^{\rho_j^2}]$ as above. Denote by $b_j$ the number of $\J_{j,i}$ at level $j$ which are not good.
By a simple large deviation bound, we have  \begin{equation}\label{ldp} \P(b_j >x_j \cdot  4^{j-1}) \le \exp \l( -4^{j-1} I_j \r) \end{equation} for any $0<x_j<1$ and  $I_j={x_j}\ln \frac{x_j}{p_j}+(1-x_j)\ln \frac{1-x_j}{1-p_j}$ (for reference, see \cite{DZ} Theorem 2.1.10).

We set $x_j=1/4^{N-j+1}$ whereas $ p_j=C\exp \l(- e^{\rho_j^2} \r)$, and \[ \rho_j=\t \log (x/2^j)=\rho_N + (N-j)\t  \] Further, $ \t \log \frac{1}{2 \la}  \le \rho_N \le \t \log \frac{1}{\la}$. Combining all these facts, we have \[-x_j\ln p_j =\l[\exp \l( \t^2 \l( \frac{1}{\t}\rho_N + (N-j) \r)^2 \r) - \ln C \r] \bigg/4^{N-j+1}\]  
By choosing $\t$ larger than and $\la$ smaller than certain absolute constants, we can make the numerator of the above expression $\ge 2\t 4^{2(N-j+1)} \text{ for all } N \ge 1 \text{ and } 1 \le j\le N$.  Since $|x_j|\le 1/4$, we have $|x_j \ln x_j| \le \frac{1}{4} \ln 4$. Also, $c_1 \le \l|(1-x_j) \ln \frac{1-x_j}{1-p_j} \r| \le c_2 \forall j$ for some positive constants $c_1$ and $c_2$. The upshot of all this is that by choosing $\t$ larger than a constant  we can make $I_j \ge \t 4^{N-j+1} \text{ for all } j$, where we recall that $I_j = {x_j}\ln \frac{x_j}{p_j}+(1-x_j)\ln \frac{1-x_j}{1-p_j} $.

Hence  we have 
\[ \P(b_j >x_j \cdot  4^{j-1}) \le \exp \l( -\t 4^{j-1}4^{N-j+1}\r)= \exp \l(- \t 4^N \r) \le \exp \l(-\t \la^2 L^2 \r) \] 

We denote by $\Omega$ the event $ \{ b_j \ge x_j \cdot  4^j \text{ for some } j  \le N   \}$.
By a union bound over $1 \le j \le N$, we have $\P (\Omega) \le N\exp \l( -{\t}{\la^2}L^2 \r) \le \exp \l( -c_2(\t)L^2  \r) $ when $\t$ is large enough, depending on $\la$.

We call $G_N^i$ to be ``good'' if each summand $g_{k,n(k,i)}^*$ in $G_i^*=\sum_{k=1}^N g_{k,n(k,i)}^*$ belong to good $\J$-s. Now, each bad $\J$ at level $j$ gives rise to $4^{N-j+1}$ bad $G_N^i$-s at level $N$. Outside the event $\Omega$, there are at most $x_j4^{j-1}$ bad $\J$-s at level $j$, leading to $x_j4^N$ bad $G_{i}^*$-s. But $\sum_{j=1}^Nx_j < 1/2$, hence except with probability $\le \exp \l( - c_2(\t) L^2\r) $, we have $\ge \frac{1}{2}4^N \ge \frac{1}{2} \la^2 L^2$ good $G_i^*$-s. For any good $G_i^*$, we have, $\t$ larger than and $\la$ smaller than  absolute constants,
\[
\label{smallerr}
\mathrm{sup}_{B_N^i}|G_i^*| \le \sum_{k=1}^N |g_{k,n(k,i)}^*| \le \sum_{k=1}^{N} e^{-\rho_k^2} \le 2 e^{-\rho_N^2} \le e^{-5\t^2}
\]
Let the final set of surviving connected square segments be denoted as $\u(B)$. 

\subsubsection{A variant of the Cantor type construction for a square}
\label{variant}
Here we will discuss a variant of the construction for the $L\t \times L\t$ square, which is as follows. We begin with a larger square $B'(L)$, each side of length $2L\t + 2 \lceil A \log L \rceil \t $. In the first step,  we remove the row and column corresponding to the central segment of length $2 \lceil A \log L \rceil \t $ on each side of the square, with $A$ as in Theorem \ref{AI}. This leaves us with four squares of side length $L\t$ each. We can parametrize the four $L\t \times L\t$ squares as $B_{00},B_{10},B_{01},B_{11}$, where an increase the first subscript denotes an increase in the horizontal co-ordinate of the centre of the square and an increase in the second subscript denotes the same for the vertical co-ordinate of the centre. We now repeat the construction for an $L\t \times L\t$ square for each of $B_{ij},0\le i,j \le 1$. Let $\u(B_{ij})$ denote the final surviving set for the construction on $B_{ij}$. The final surviving set  for $B'(L)$ is denoted $\u=\cup_{i,j=0}^{1}\u(B_{ij})$. The normalized GAF $f^*$ on $B'(L)$ has the analogous decomposition on similar lines as the $L\t \times L\t$ square case. Further, with probability $\le C\exp(-e^{(\t\log L)^2})$, all the errors $g^*$ due to the first step of the construction (from $B'(L)$ to $B_{ij}$) are $\le e^{-(\t\log L)^2}$. 

Hence, on similar lines to Section \ref{func}, we can conclude that for $\la$ smaller than an absolute constant and $\t$ large enough (depending on $\la$), in each configuration $\u(B_{ij})$ at least $\frac{1}{2}4^{N} \ge \frac{1}{2} \la^2 L^2$ surviving components are good, except on an event $\Omega$ with $ \P\l(\Omega \r) \le  \exp \l( -c(\t)L^2  \r) $. On each good component in the final surviving set, the accumulated error $G^*$ satisfies $\text{max}|G^*|\le e^{-5\t^2}$.

\subsubsection{Proof of Theorem \ref{gafest}}
\label{pfthm}
Suppose we have a connected set $\g$ of standard squares of base length $\t$ and consisting of $L$ standard squares. Then there is a square $B$ of side length $L\t$, consisting of $L^2$ standard squares of base length $\t$, such that $\g \subset B$. As in Section \ref{variant},  we form a square $B'(L)$ of side length $2L\t + 2\lceil A\log L \rceil \t$, consisting of standard squares, such that $B$ sits inside $B'(L)$ as the square $B_{11}$.  Let the final Cantor set on $B'(L)$ be $\u=\cup_{i,j=0}^{1}\u(B_{ij})$. Recall that the connected components of $\u$ are squares of side length $\ge \t/4\la$, and each such component contains at least one standard square of side $\t$. Moreover, any two the $\u(B_{ij})$s are isometrically isomorphic with each other under an appropriate horizontal and/or vertical translation by $(L+2\lceil A\log L \rceil) \t$. We select one such standard square from each connected component in $\u(B_{11})$, and in the other $\u(B_{ij})$s we select  those standard squares which correspond under the above isomorphisms to the ones chosen in $\u(B_{11})$. The resulting union of standard squares for each $B_{ij}$ will be denoted by $\u'(B_{ij})$ and we define $\u'=\cup \u'(B_{ij})$. So, Given any standard square $\sigma$ in $ \u'( B_{11})$, there are four isomorphic copies corresponding to  $\sigma$  in $\u'$ (under the translations mentioned above, one copy in each $\u'(B_{ij})$). 

Denote by $\u'+(m,n)$ the translate in $\R^2$ of the set $\u'$  by the vector $(m\t,n\t)$. Let $\I$ be the set of such translates of $\u'$ by $(m,n)$ in the range $0 \le m,n \le L+ 2\lceil A \log L \rceil$. Given any standard square $\sigma$ in $ \u'( B_{11})$, there are four isomorphic copies corresponding  $\sigma$ in $\u'$  and it is easy to check under the action of $\I$, any given square $\in \g$ is covered by exactly one of these. The total number of such $\sigma$ (not counting  isomorphic copies) is $ 4^N \ge \la^2 L^2$, while $|\I|=(2L + 2\lceil A \log L \rceil)^2$. Hence, choosing a translate from $\I$ uniformly at random, the probability that a particular standard square in $ \g$ is covered $\ge \la^2 L^2/(2L + 2\lceil A \log L \rceil)^2 \ge \kappa= \kappa(\la)>0$.  Hence the expected fraction of $\g$ covered by a random translate of $\u'$ is also $\ge \k$. This implies that there exists a translate $T(\u')$ of $\u'$ (chosen from $\I$) such that it covers at least a fraction $\kappa$ of the standard squares in $\g$. The same translation gives rise to a translate $T(\u)$ of $\u$. 

We call a constituent standard square of $T(\u')$ to be ``good'' if the corresponding $G_{k}^*$ in $T(\u(B_{ij}))$ (from the Cantor type construction applied to the square $T(B'(L))$) is good as defined in Section \ref{func}. But in each $T(\u(B_{ij}))$, we have  that except on an event $\Omega$ such that $\P(\Omega)<e^{-c(\t)L^2}$, we have at least $\frac{1}{2}$ of the  $G_k^*$-s to be good. Let $\g(i),1\le i \le L$ denote the standard squares in $\g$. Call a standard square to be ``empty'' or ``full'' according as  it contains respectively 0 or $\ge k$ points in $\F$.  Call $\g$ ``empty'' or ``full'' if all standard squares in $\g$ are empty or full. In what follows, we treat the state ``empty'', but in all steps it can be replaced by the state ``full''. 

We observe that \[ \{ \g \text{ is empty }  \} \subset \Omega \cup \{ \text{ Some subset of } \lfloor \k L \rfloor \text{ standard squares in } T(\u') \cap \g \text{ are good and empty } \} \]

We have, via a union bound,
\[\P(\g \text{ is empty} ) \le \P(\Omega) + \sum_{\mathcal{S}} \P \l( \bigcap_{\{ \g_{i_k} \} \in \mathcal{S}}   \{ \g_{i_k}  \text{ is empty and good }   \} \r) \]     
where the last summation is over $\mathcal{S}$ which is the collection of all possible subsets $\{ \g_{i_k} \}$ of $\lfloor \k L \rfloor$ standard squares in $\g$ such that $\g_{i_k} \in T(\u') \text{ for all } k$.
Since there are at most $2^L$ such subsets, it suffices to show that for any fixed  $\{ \g_{i_k}\} \in \mathcal{S}$, we have for large enough $\t$
\begin{equation} \label{target1} \P \l( \bigcap_{\{ \g_{i_k} \} \in \mathcal{S}}   \{ \g_{i_k}  \text{ is empty and good }   \} \r) \le \exp \l( -c_1(\t)L \r) 
\end{equation} 
\begin{equation}
 \label{target2}
\P \l( \bigcap_{\{ \g_{i_k} \} \in \mathcal{S}}   \{ \g_{i_k}  \text{ is full and good }   \} \r) \le \exp \l( -c_2(\t,k)L \r)
\end{equation}

where $c_1(\t) \to \infty$ as $\t \to \infty$ and $\lim_{\t \to \infty} \lim_{k \to \infty}c_2(\t,k)=\infty$.

Let $A_{i_k}$ denote the event that $   \{ \g_{i_k}  \text{ is empty and good } \} $. Recall that $\g_{i_k}$ being empty implies that $f^*|_{\g_{i_k}}$ does not have any zeros, and $\g_{i_k}$ being good implies that $\max_{\g_{i_k}} |G_{i_k}^*| \le e^{-c\t^2}$, where $G_{i_k}^*$ are the cumulative errors in the cantor set construction, as estimated in Section \ref{func}. 

Define $A_{i_k}'$ to be the event that $f_{i_k}^*|_{\g_{i_k}}$ does not have any zeros. Here $f_{i_k}^*$ are the final independent normalized GAFs obtained in the Cantor set construction. Clearly, the events $A_{i_k}'$ are independent.

We will show that $A_{i_k} \subset A_{i_k}' \cup \Omega_{i_k}$, where the  $\Omega_{i_k}$-s are independent events with  $\P(\Omega_{i_k})<e^{-c\t}$. To this end, we note that on $\g_{i_k}$, we have $f^*=f_{i_k}^*+G_{i_k}^*$, and also $\max_{\g_{i_k}} |G_{i_k}^*| \le e^{-5\t^2}$. Applying Corollary \ref{estim2} to the square $\g_{i_k}$, we deduce that except for a bad event $\Omega_{i_k}$ of probability $\le e^{-c\t}$, we have $|f_{i_k}^*|>e^{-5\t^2}$ on $\partial \g_{i_k}$. Hence the equation $f^*=f_{i_k}^*+G_{i_k}^*$ on $\g_{i_k}$ along with Rouche's theorem implies that $f^*$ and $f_{i_k}^*$ have the same number of zeros in $\g_{i_k}$. So, on $A_{i_k} \cap \Omega_{i_k}^c$ we have that $A_{i_k}'$ holds, in other words $A_{i_k} \subset A_{i_k}' \cup \Omega_{i_k}$, as desired. The $\Omega_{i_k}$-s are independent since $\Omega_{i_k}$ is defined in terms of $f^*_{i_k}$ which are independent normalized GAFs.

Therefore we can write, for a fixed $\{\g_{i_k}\} \in \mathcal{S}$
\[\P \l( \bigcap_{k} A_{i_k} \r) \le \P \l( \bigcap_{k} \l( A_{i_k}' \cup \Omega_{i_k} \r) \r) \le \prod_{k} \P \l( A_{i_k}' \cup \Omega_{i_k} \r) \] 
But it is not hard to see that for the state ``empty'' we have $\P \l( A_{i_k}' \cup \Omega_{i_k} \r) \le \P\l( A_{i_k}' \r) + \P\l( \Omega_{i_k} \r) \le e^{-c(\t)}$ where $c(\t) \to \infty$ as $\t \to \infty$. It is also easy to see that if we consider the state ``full'' instead of ``empty'' $\P( A_{i_k}') \le e^{-c(\t,k)}$ where $c(\t,k) \to \infty$ as $k \to \infty$ for fixed $\t$, and $\P(\Omega_{i_k}) \le e^{-c\t}$. Therefore we have $\P \l( A_{i_k}' \cup \Omega_{i_k} \r) \le \exp \l( -c_2(\t,k) \r)$ where $\lim_{\t \to \infty} \lim_{k \to \infty}c_2(\t,k) = \infty$. This proves equations (\ref{target1}) and (\ref{target2})  and hence completes the proof of the theorem. 

\subsubsection{Lower bound on the size of $f^*$ }
\label{lbound}
Our goal in this section is to establish that with large probability, the size of a normalized GAF on the perimeter of a circle (or a square) cannot be too small. Of course, there is a trade-off between the ``largeness'' of the probability and ``smallness'' of the GAF, depending on the radius of the circle or the side length of the square. Such estimates, along with Rouche's theorem, would be useful in replacing $f^*|_{\g_{i_k}}$ with the independent $f^*_{i_k}$ on ``good'' $\g_{i_k}$-s in Section \ref{pfthm}.

\begin{proposition}
\label{estim1}
Let us consider a disk $\D$ of radius $R>1$, and let $\nu>2$. Then \begin{equation} \label{disk} \P \l( |f^*(z)| \le e^{-\nu R^2} \text{ for some } z \in \partial \D \r) \le e^{-C(\nu)R} \end{equation} for a $C(\nu)>0$. Here $f^*(z)=e^{-\frac{1}{2}|z|^2}f(z)$ where $f$ is the standard planar GAF.   
\end{proposition}

\begin{proof}
First, we show that we can take $\D$ to be centred at the origin. Let $w$ be the centre of $\D$ and let $\D'$ be the disk of the same radius as $\D$ centred at the origin.  Then the random field $|f^*(z)|$ for $z \in \D$ can be re-parametrized as the random field $\exp \l( -\frac{1}{2}|z+w|^2 \r)|f(z+w)|$ for $z \in \D'$. The latter can be written as  $e^{-\frac{1}{2}|z|^2} \exp \l( -\Re (z\overline{w})-\frac{1}{2}|w|^2  \r)|f(z+w)|$. But it is clear from a  simple covariance computation that for any fixed $w$ the random fields $f(z)$ and $f_{w}(z)=\exp \l( -z\overline{w}-\frac{1}{2}|w|^2  \r)f(z+w)$  have the same distribution, and so do $|f(z)|$ and $|f_w(z)|=\exp \l( -\Re (z\overline{w})-\frac{1}{2}|w|^2  \r)|f(z+w)|$. In other words, the random field $\{\l|f^*(z) \r|\}_{z \in \C}$ has a translation-invariant distribution. Hereafter we assume that $\D$ is centred at the origin.

We want to show that $|f^*(z)|\ge e^{-\nu R^2}$, or equivalently, $|f(z)| \ge e^{-(\nu-1/2) R^2}$ on $\partial \D$ except on an event with probability exponentially small in $R$. We choose $\eta=\lceil 2\pi R^3 e^{(\nu +2 ) R^2} \rceil$ equi-spaced points $\{z_j\}_{j=1}^{\eta}$ on $\partial \D$. Then \[ \P \l( |f^*(z_j)| \le 2e^{-\nu R^2} \r) \le ce^{-2\nu R^2} \text{ for all } j\] since $f^*(z)$ is a standard complex Gaussian  for each fixed $z$. Consider the event   $\Omega_1= |f^*(z_j)| > 2e^{-\nu R^2} \text{ for all } j \le N$. By a union bound over the $\{z_j \}$-s we have   $\P(\Omega_1^c )\le ce^{-(\nu-2) R^2} $ for some $c(\nu)>0$. 

$f'(z)$ is a centred analytic Gaussian process on $\C$ with covariance  $\mathrm{Cov}(f'(z),f'(w))=z\overline{w}e^{z\overline{w}}$. Set $\sigma_{R}^2=\text{max}_{2 \cdot \D}\mathrm{Var}\l(f'(z)\r)=4R^2e^{4R^2}$. We use Lemma 2.4.4 from \cite{HKPV} and apply it to the Gaussian analytic function $f'(2Rz)$. Using this, we obtain \[ \P \l( \text{Max}_{\D} |f'(z)|>t \r) \le 2\exp \l( -\frac{t^2}{8 \sigma_{R}^2} \r) \] 
Setting $t=R^{3/2}e^{2R^2}$ in the above we get \[ \P \l( \text{Max}_{\D} |f'(z)|>R^{3/2}e^{2R^2} \r) \le 2e^{ -\frac{1}{32 }R } \] Let  $\Omega_2$ denote the event $\{\text{Max}_{\D} |f'(z)| \le R^{3/2}e^{2R^2} \}$.
The distance between any two consecutive $z_j$-s is $\le 2\pi R/\eta \le R^{-2} e^{-(\nu+2) R^2}$. Hence, on $ \Omega_2$ we have, via the mean value theorem, $|f(z)-f(z_j)|<  R^{- 1/2} e^{-\nu R^2}$ for any point $z \in \partial \D$ where $z_j$ is the nearest point to $z$ among $\{z_j\}_{j=1}^{\eta}$. As a result, for $R>1$ we have on $\Omega_1 \cap \Omega_2$ \[|f(z)|\ge |f(z_j)|-|f(z) -f(z_j)|\ge 2e^{-(\nu -1/2) R^2} -  R^{-1/2} e^{-\nu R^2} \ge e^{-(\nu -1/2) R^2} \]
For $R>1$ we have $\P(\Omega_1 \cap \Omega_2) \ge 1 - e^{-C(\nu)R}$ for some constant $C(\nu)>0$, as desired.
\end{proof}

\begin{corollary}
 \label{estim2}
Let us consider a square $\Tau$ of side length $S>1$, and let $\nu>1$. Then \begin{equation} \label{disk} \P \l( |f^*(z)| \le e^{-\nu S^2} \text{ for some } z \in \partial \Tau \r) \le e^{-C(\nu)S} \end{equation} for some constant $C(\nu)>0$. Here $f^*(z)=e^{-\frac{1}{2}|z|^2}f(z)$ where $f$ is the standard planar GAF.   
\end{corollary}

\begin{proof}
We can follow  the proof of Proposition  \ref{estim1}. The maximum of $|f'|$ on $\Tau$ is bounded above by the maximum of $|f'|$ on the circumscribing circle of $\Tau$, whose radius is $S/\sqrt{2}$ and for which we can apply Lemma 2.4.4 from \cite{HKPV}. 
\end{proof}

\subsection{Proof of Theorem \ref{gaf}: existence of critical radius}
\label{gafperc}
We simply observe that Theorem \ref{gafest} proves that the criteria outlined in Proposition \ref{tech} are valid for $\F$, thereby establishing that a critical radius exists for $\F$.

\section{Uniqueness of infinite cluster}
\label{uinf}
In this section we will prove that in the supercritical regime for the Boolean percolation models $(\G,r)$ and $(\F,r)$, a.s. there is exactly one infinite cluster.
\subsection{An approach to uniqueness}
\label{criterion}
We will first describe a proposition which has important implications regarding such uniqueness  for a translation invariant point process $\Pi$. 
\begin{proposition}
\label{critprop}
Let $r>r_c$ for the Boolean percolation model $X(\Pi,r)$, where $\Pi$ is a translation invariant point process on $\R^2$, and $0<r_c<\infty$ is the critical radius. For $R>0$ let $B_R$ denote the set $\{ x \in \R^2 : \|x\|_{\infty} \le R \}$. Suppose the following event has positive probability:
\[ E(R)= \bigg\{ \text{ There is an infinite cluster } C' \text{ with the property that } C'\cap (B_R)^c \text{ contains } \]\[ \text{ at least three infinite clusters and such that there is at least one point from } \Pi \text{ in } C'\cap B_R   \bigg\} \]
Then $\P\l(E(R)\r)=0$.
\end{proposition}
 
The proof of the above proposition can be found in \cite{MR}, proof of Theorem 3.6. The event $E(R)$ from Proposition \ref{critprop} corresponds to the event $E^0(N)$ there.

A general approach to a proof that a.s. there cannot be inifnitely many infinite clusters is to show that such an event would imply $E(R)$ would occur for some $R$.

\subsection{Uniqueness of infinite clusters: Ginibre ensemble}
\label{uqgin}
In this section we prove that in $X(\G,r)$ with $r>r_c$, we have $\L(r)=1$ a.s.

To this end, we would need to have an understanding of the conditional distribution of the points of $\G$ inside a domain given the points outside. This has been obtained in \cite{GNPS} Theorems 1.1 and 1.2. We state these results below.

Let $\D$ be a bounded open set in $\C$ whose boundary has zero Lebesgue measure, and let $\S_{\iin}$ and $\S_{\out}$ denote the Polish spaces of locally finite point configurations on $\D$ and $\D^c$ respectively. $\G_{\iin}$ and $\G_{\out}$ respectively denote the point processes obtained by restricting $\G$ to $\D$ and $\D^c$.  
\begin{theorem}
 \label{gin-1}
For the Ginibre ensemble, there is a measurable function $N:\S_{\out} \to \nat$ such that a.s. \[ \text{ Number of points in } \G_{\iin} = N(\G_{\out})\hspace{3 pt}.\]
\end{theorem}

Let the points of $\G_{\iin}$, taken in uniform random order, be denoted by the vector $\uz$. Let $\rho(\uout,\cdot)$ denote the conditional measure of $\uz$ given $\G_{\out}=\uout$. Since a.s. the length of $\uz$ equals $N(\G_{\out})$, we can as well assume that each measure $\rho(\uout,\cdot)$ is supported on $\D^{N(\uout)}$.

\begin{theorem}
\label{gin-2}
For the Ginibre ensemble, $\P[\G_{\out}]$-a.s. $\rho({\G_{\out}},\cdot)$ and the Lebesgue measure $\el$ on $\D^{N(\G_{\out})}$ are mutually absolutely continuous.
\end{theorem}

We are now ready to prove Theorem \ref{gin}.

\begin{proof} [\textbf{Proof of Theorem \ref{gin}}]
Let $r$ be such that $\L(r)>0$ a.s.
In what follows, we will repeatedly use the fact that if there are two points $x,y \in \R^2$ at Euclidean distance $d$, then there can be connected to each other by $(1+\lceil {d/2r} \rceil)$  open disks of radius $r$, such that no two disks are exactly at a distance $r$.  

We will first deal with the case where a.s. $\L(r)>1$ but finite. A similar argument will show that if $3 \le \L(r) \le \infty$ then the event $E(R)$ as in Proposition \ref{critprop} occurs, with a suitable choice of $R$. This would rule out the possibility $\L(r)=\infty$, and complete the proof.  

We argue by contradiction, and let if possible $1<\L(r)<\infty$ a.s. 
Let  $\D_1 \subset \D_2$ be two concentric open disks centred at the origin and respectively having radii $R_1 < R_2$. 

Let $\e$ be the event that:
\begin{itemize}
 \item (i) There are two infinite clusters $\c_1$ and $\c_2$ in the underlying graph $\j$ such that $ \c_1 \cap \D_1 \ne \emptyset \ne \c_2 \cap \D_1$.
 \item (ii) There exists a finite cluster $\c_3$ of vertices of $\j$  which has $\ge 1+ \lceil 2 R_1/r \rceil $  vertices such that $\c_3 \subset \text{ Int} \l(\D_2 \setminus \D_1 \r) $, where $\text{ Int}(A)$ is the interior of the set $A$.
\end{itemize}

By ergodicity of $\G$, $\P(\e)>0$ when $R_1$ and $R_2$ are large enough. Fix such disks $\D_1$ and $\D_2$. We denote the configuration  of points outside $\D_2$ by $\o$ and those inside $\D_2$ by $\z$. Let the number of points in $\D_2$ be denoted by $N(\o)$. Any two points of $\Pi$ inside $\D_1$ are at most at a Euclidean distance of $2R_1$, and  hence can be connected by at most $\l(1+ \lceil 2R_1/r \rceil \r)$ open disks of radius $r$ such that no two points are exactly at a distance $r$. We define an event $\e'$ as follows: corresponding to every configuration $(\z,\o)$ in $\e$, we define a new configuration $(\z',\o)$ where $\z'$ is obtained by moving  $\l( 1+ \lceil 2R_1/r \rceil \r)$ points of $\c_3$ to the interior of $\D_1$ and placing them such that in  the new underlying  graph $\j'$ (for definition see Section \ref{ug}) the clusters $\c_1$ and $\c_2$ become connected with each other. Similar to the observations made in Section \ref{ug}, we can move each point in $\z'$ in sufficiently small disks around itself, resulting in new configurations $(\z'',\o)$ such that the connectivity properties of $\j'$ as well as the number of points in $\D_2$ remain unaltered. The event $\e'$ consists of all such configurations $(\z'',\o)$ as $(\z,\o)$ varies over all configurations in $\e$. Observe that for each $\o$, the set of configurations $\{\z'': (\z'',\o) \in \e' \}$ constitutes an open subset of $\D^{N(\o)}$, when considered as a vector in the usual way.  Since $\P(\e)>0$, by Theorem \ref{gin-2} applied to the domain $\D_2$, we also have $\P(\e')>0$. But on $\e'$, there is one less infinite cluster than on $\e$. This gives us the desired contradiction, and proves that $\P(1<\L(r)<\infty)=0$.

Had it been the case $\L(r)\ge 3$ a.s., observe that  an argument analogous to the previous paragraph can be carried through with three instead of two  infinite clusters ($\c_1$ and $\c_2$ above). The end result would be that with positive probability we can connect all the three clusters with each other.
If $\L(r)=\infty$ a.s. then we carry out the above argument with three of the infinite clusters, and observe that the event $E(R)$ as in Proposition \ref{critprop} occurs with a set $B_R$ where $R > R_2$, on the modified event (analogous to $\e'$ above). This proves that $\P(\L(r)=\infty)=0$.
\end{proof}

\subsection{Uniqueness of infinite clusters: Gaussian zeroes}
\label{uqgaf}
In this section we prove that in $X(\F,r)$ with $r>r_c$, we have $\L(r)=1$ a.s.

To this end, we would need to have an understanding of the conditional distribution of the points of $\F$ inside a domain given the points outside. This has been obtained in \cite{GNPS} Theorems 1.3 and 1.4. $\F_{\iin}$ and $\F_{\out}$ respectively denote the point processes obtained by restricting $\F$ to $\D$ and $\D^c$ respectively.  We state these results below. Some of the notation is from Section \ref{uqgin}.

\begin{theorem}
 \label{gaf-1}
For the GAF zero ensemble, \newline
\noindent
(i)There is a measurable function $N:\S_{\out} \to \nat$ such that a.s.  \[ \text{ Number of points in } \F_{\iin} = N(\F_{\out}).\]
(ii)There is a measurable function $S:\S_{\out} \to \C$ such that a.s.  \[ \text{ Sum of the points in } \F_{\iin} = S(\F_{\out}).\]
\end{theorem}

Define the set \[ \Sigma_{S(\F_{\out})} := \{ \uz \in \D^{N(\F_{\out})} : \sum_{j=1}^{N(\F_{\out})} \z_j = S(\F_{\out}) \}\] where $\uz=(\z_1,\cdots,\z_{N(\F_{\out})})$.

Since a.s. the length of $\uz$ equals $N(\F_{\out})$, we can as well assume that each measure $\rho(\uout,\cdot)$ gives us the distribution of a random vector in  $\D^{N({\uout})}$ supported on $\Sigma_{S({\uout})}$.

\begin{theorem}
\label{gaf-2}
For the GAF zero ensemble, $\P[\F_{\out}]$-a.s. $\rho(\F_{\out},\cdot)$ and the Lebesgue measure $\els$ on $\Sigma_{S(\F_{\out})}$ are mutually absolutely continuous.
\end{theorem}

We are now ready to prove Theorem \ref{gaf}.

\begin{proof}[\textbf{Proof of Theorem \ref{gaf} : uniqueness of infinite cluster }] 
The proof follows the contour of Section \ref{uqgin}, with extended arguments to  take care of the fact that for $\F$ there are two conserved quantities for  local perturbations of the zeros inside a disk : their number and their sum, unlike $\G$ where only the number of points is conserved. 

We first show that it cannot be true that a.s. $1 < \L(r) < \infty$. We argue by contradiction, and let if possible $1<\L(r)<\infty$ a.s. 
We will define events $\e$ and $\e'$ in analogy to Proposition \ref{uqgin} such that on $\e'$ there one less infinite cluster than on $\e$ and $\P(\e)>0$ and $\P(\e')>0$. 

Let  $\D_1,\D_2$ and $\D_3$ be two concentric open disks centred at the origin and respectively having radii $R_1 < R_2 < R_3$. 

Let $\e$ be the event that : \newline
(i)$ \c_1 \cap \D \ne \emptyset \ne \c_2 \cap \D $ for two infinite clusters $\c_1$ and $\c_2$ \newline
(ii)$\exists$ a cluster $\c_3$ of vertices of the underlying graph $\j$ which has $\ge n = 1+ \lceil 2R_1/r \rceil $  vertices such that $\c_3 \subset \text{ Int } \l(\D_2 \setminus \D_1 \r) $ \newline
(iii)$\exists$ a cluster $\c_4 \subset \text{ Int } \l(\D_3 \setminus \D_2 \r)$ with $\ge n'= 2R_2n$ vertices such that $\text{ dist } (\c_4, \j \setminus \c_4) > 10$ (recall that $\j$ denotes the underlying graph, for definition see Section \ref{ug}).

By the ergodicity of $\F$, $\P(\e)>0$ when $R_i,i=1,2,3$  are large enough. Fix such disks $\D_i,i=1,2,3$. We denote the configuration of points outside $\D_3$ by $\o$ and those inside $\D_3$ by $\z$. Let the number of points in $\D_3$ be denoted by $N(\o)$ and let their sum be $S(\o)$. 

We start with a configuration $(\z,\o)$ in $\e$. We perform the same operations as in Proposition \ref{uqgin} with the points inside $\D_2$. However, in $\F$, unlike in $\G$, we need to further ensure that the sum of the points inside $\D_3$ remain unchanged at $S(\o)$. We note that due to the operations already performed on the points inside $\D_2$, the sum of the points inside $\D_3$ has changed by at most $2R_2n$, since $\le  n$ points have been moved and each of them can move by at most $2R_2$ which is the diameter of $\D_2$. We observe that we can compensate for this  by translating each point in $\c_4$ by an amount $\le 1$ in an appropriate direction. Due to the separation condition in (iii) in the definition of $\e$, this does not change the connectivity properties of any vertex in $\j \setminus \c_4$. 

By the observations made in Section \ref{ug},  we can move each point in $\z'$ in sufficiently small disks around itself, resulting in new configurations $(\z'',\o)$ such that the connectivity properties of $\j'$ as well as the number of the points in $\D_3$ remain unaltered. The event $\e'$ consists of all such configurations $(\z'',\o)$ as $(\z,\o)$ varies over all configurations in $\e$. Observe that for each $\o$, the set of configurations $\{\z'': (\z'',\o) \in \e' \}$ constitutes an open subset of $\D_3^{N(\o)}$, when considered as a vector in the usual way, hence its intersection with $\Sigma_{S(\o)}$ is an open subset of $\Sigma_{S(\o)}$. Since $\P(\e)>0$, by Theorem \ref{gaf-2} applied to the domain $\D_3$, we also have $\P(\e')>0$. But in $\e'$, there is one less  infinite cluster than in $\e$. This gives us the desired contradiction, and proves that $\P(1<\L(r)<\infty)=0$.

We take care of the case $\L(r)=\infty$ as we did in the proof of Theorem \ref{gin}. Had it been the case $\L(r)\ge 3$ a.s.,  an argument analogous to the previous paragraph can be carried through with three instead of two  infinite clusters ($\c_1$ and $\c_2$ above), with the end result that with positive probability we can connect all the three infinite clusters with each other. If $\L(r)=\infty$ a.s. then we carry out the above argument with three of the infinite clusters, and observe that the event $E(R)$ in Proposition \ref{critprop} occurs on the modified event (analogous to $\e'$ above) with a set $B_R$ where $R > R_3$. This proves that $\P(\L(r)=\infty)=0$.
\end{proof}

\textbf{Acknowledgements.}
We are grateful to Fedor Nazarov for suggesting the approach to Theorem \ref{gafest} and helpful discussions.

\end{document}